\documentclass[]{scrartcl}

\usepackage[]{amsmath, amssymb,amsfonts,amsthm,mathtools,braket,color,enumerate}
\usepackage{bbding} 

\usepackage{here}
\usepackage{tikz-cd}
\usepackage{tikz}
\usetikzlibrary{decorations}
\usetikzlibrary{arrows}
\usetikzlibrary{hobby}
\tikzstyle{v} = [circle, draw, inner sep=2pt, minimum size=3pt, fill=black]

\makeatletter

\let\c@table\c@figure
\makeatother 

\theoremstyle{plain}
\newtheorem{theorem}{Theorem}[section]
\newtheorem{lemma}[theorem]{Lemma}
\newtheorem{proposition}[theorem]{Proposition}
\newtheorem{corollary}[theorem]{Corollary}

\theoremstyle{definition}
\newtheorem{definition}[theorem]{Definition}

\newtheorem{example}[theorem]{Example}

\newtheorem{remark}[theorem]{Remark}

\DeclareMathOperator{\UBG}{UBG}

\title {Unit Ball Graphs on Geodesic Spaces}
\author{
Masamichi Kuroda \and
Shuhei Tsujie\thanks{Corresponding author}}
\date{}

\begin{document}
\maketitle
	
\begin{abstract}
Consider finitely many points in a geodesic space. 
If the distance of two points is less than a fixed threshold, then we regard these two points as ``near". 
Connecting near points with edges, we obtain a simple graph on the points, which is called a unit ball graph. 
If the space is the real line, then it is known as a unit interval graph. 
Unit ball graphs on a geodesic space describe geometric characteristics of the space in terms of graphs. 
In this article, we show that every unit ball graph on a geodesic space is (strongly) chordal if and only if the space is an $ \mathbb{R} $-tree and that every unit ball graph on a geodesic space is (claw, net)-free if and only if the space is a connected manifold of dimension at most $ 1 $. 
As a corollary, we prove that the collection of unit ball graphs essentially characterizes the real line and the unit circle. 
\end{abstract}

{\footnotesize \textit{Keywords}: 
geodesic space, 
$ \mathbb{R} $-tree, 
real tree, 
$ 0 $-hyperbolic space,
unit ball graph, 
unit interval graph, 
intersection graph, 
chordal graph, 
strongly chordal graph, 
$ (\text{claw}, \text{net}) $-free graph, 
Hamiltonian-hereditary graph
}

{\footnotesize \textit{2020 MSC}: 
51D20, 
05C62, 
05C75 
}

\section{Introduction}
Let $ (X, d) $ be a metric space. 
Consider finitely many points $ x_{1}, \dots, x_{n} $ in $ X $ and fix a threshold $ \delta > 0 $. 
If $ d(x_{i}, x_{j}) \leq \delta $, we regard $ x_{i} $ and $ x_{j} $ as ``near". 
We can construct a simple graph on the set $ \{x_{1}, \dots, x_{n}\} $ with edges between near points. 
It might be expected that we could obtain some information about $ X $ from graphs constructed in such a way. 
However, it seems difficult to study metric spaces with unit ball graphs without any other assumptions. 
For example, let $ (X,d) $ be a metric space defined by 
\begin{align*}
X \coloneqq \Set{(\cos \theta, \sin \theta) \in \mathbb{R}^{2} | \dfrac{\pi}{6} \leq \theta \leq \dfrac{11\pi}{6}}
\end{align*}
and $ d $ is the restriction of the Euclidean metric in $ \mathbb{R}^{2} $. 
Take points $ x = (\cos\frac{\pi}{6}, \sin\frac{\pi}{6}), y = (\cos\frac{11\pi}{6}, \sin\frac{11\pi}{6}) $, and $ z = (-1,0) $ and suppose that $ \delta = 1 $ (see Figure \ref{Fig:C}). 
\begin{figure}[t]
\centering
\begin{tikzpicture}
\draw (.865,.5) arc (30:330:1);
\draw (.865, .5) node[right](x){$ x $}; 
\draw (.865,-.5) node[right](y){$ y $}; 
\draw (-1,0) node[left](z){$ z $};
\end{tikzpicture}
\caption{Intuitively $ x $ and $ y $ should be furthest but not with the Euclidean metric} \label{Fig:C}
\end{figure}
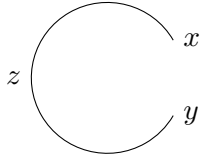
Then $ x $ is ``near" to $ y $ but not to $ z $, which seems counterintuitive. 
It is natural to regard $ x $ and $ y $ as the ``furthest" points in $ X $ and $ z $ is the midpoint between $ x $ and $ y $. 
If we define a metric on $ X $ by arc length, then it fits our intuition. 
Thus, from now on, our interest focus on geodesic spaces defined as follows. 
\begin{definition}
Let $ (X,d) $ be a metric space and $ x,y \in X $. 
A \textbf{geodesic} from $ x $ to $ y $ is a distance-preserving map $ \gamma $ from a closed interval $ [0, d(x,y)] \subseteq \mathbb{R} $ to $ X $ with $ \gamma(0)=x $ and $ \gamma(d(x,y)) = y $. 
Its image is said to be a \textbf{geodesic segment} with endpoints $ x $ and $ y $. 
We say that $ (X,d) $ is a \textbf{geodesic space} if every two points are joined by a geodesic. 
Note that a geodesic segment between two points is not necessarily unique. 
We will write a geodesic segment whose endpoints are $ x $ and $ y $ as $ [x,y] $. 
If there exists a unique geodesic segment for every pair of points, then we say that $ (X,d) $ is \textbf{uniquely geodesic}.   
\end{definition}

\begin{example}
The $ n $-dimensional Euclidean space $ \mathbb{R}^{n} $ and its convex subsets are geodesic spaces. 
The $ n $-dimensional sphere $ S^{n} $ with the great-circle metric is a geodesic space. 
The vector space $ \mathbb{R}^{n} $ with $ L_{p} $-norm $ || \ast ||_{p} $ ($ 1 \leq p \leq \infty $) is also a geodesic space. 
\end{example}

Next we define the graphs in which we are interested. 
\begin{definition}
Let $ (X, d) $ be a (geodesic) metric space. 
A simple graph $ G = (V_{G}, E_{G}) $ is said to be a \textbf{unit ball graph} on $ (X,d) $ if there exist a \textbf{threshold} $ \delta > 0 $ and a map $ \rho $, called a \textbf{realization}, from the vertex set $ V_{G} $ to $ X $ such that $ \{u,v\} \in E_{G}$ if and only if  $ d(\rho(u), \rho(v)) \leq \delta $. 
Let $ \UBG(X, d) $ denote the collection of the unit ball graphs on $ (X, d) $. 
When there is no confusion with the metric, we may write it as $ \UBG(X) $. 
\end{definition}

\begin{remark}
In this article, the term ``graph" refers an undirected simple graph on finite vertices.  
We frequently identify the vertices of a unit ball graph with the realized points in the space. 
\end{remark}

\begin{remark}
A unit ball graph is the intersection graph of finitely many closed balls of the same size in a geodesic space. 
If we scale the metric, then the graph can be the intersection graph of unit balls, that is, balls of radius $ 1 $. 
When we consider the Euclidean spaces, we may always assume that a unit ball graph is the intersection graph of finitely many unit balls. 
\end{remark}

Let $ H $ be a graph. 
A graph is said to be \textbf{$ H $-free} if it has no induced subgraph isomorphic to $ H $. 
A graph is called \textbf{chordal} if it is $ C_{n} $-free for all $ n \geq 4 $, where $ C_{n} $ denotes the cycle graph on $ n $ vertices. 
Note that a graph $ G $ is chordal if and only if every cycle of $ G $ of length four or more has a \textbf{chord}, which is an edge connecting non-consecutive vertices of the cycle. 
For an integer $ n \geq 3 $, the (complete) \textbf{$ n $-sun} is a graph on $ 2n $ vertices $ \Set{v_{i} | i \in \mathbb{Z}/2n\mathbb{Z}} $ such that the even-indexed vertices induce a complete graph, the odd-indexed vertices form an independent set, and an odd-indexed vertex $ v_{i} $ is adjacent to an even-indexed vertex $ v_{j} $ if and only if $ i - j =\pm 1 $ (See Figure \ref{Fig:7-sun} for example). 
A graph is called \textbf{sun-free} if it is $ n $-sun-free for all $ n \geq 3 $. 
A graph is called \textbf{strongly chordal} if it is chordal and sun-free. 
Farber \cite{farber1983characterizations-dm} investigated strongly chordal graphs and gave some characterizations. 
The definition above is one of such characterizations. 

\begin{figure}[t]
\centering
\begin{tikzpicture}
\foreach \a in {1,3,5,7,9,11,13}{
\draw (\a*360/14-360/14+90: 17.3mm) node[v](\a){};
}
\foreach \a in {2,4,6,8,10,12,14}{
\draw (\a*360/14-360/14+90: 10mm) node[v](\a){};
}
\foreach \u / \v in{1/2,2/3,3/4,4/5,5/6,6/7,7/8,8/9,9/10,10/11,11/12,12/13,13/14,14/1,2/4,2/6,2/8,2/10,2/12,2/14,4/6,4/8,4/10,4/12,4/14,6/8,6/10,6/12,6/14,8/10,8/12,8/14,10/12,10/14,12/14}{
\draw (\u) -- (\v);
}
\end{tikzpicture}
\caption{$ 7 $-sun} \label{Fig:7-sun}
\end{figure}

A graph in $ \UBG(\mathbb{R}) $ is called a \textbf{unit interval graph} also known as an \textbf{indifference graph}, which is a very important object in combinatorics. 
There are several linear-time algorithms for recognizing unit interval graphs. 
Also, they are characterized by forbidden induced subgraphs as explained below. 

\begin{theorem}[Wegner \cite{wegner1967eigenschaften}, Roberts \cite{roberts1969indifference-ptigt}]\label{WR unit interval}
A graph is unit interval if and only if it is chordal and $ (\text{claw}, \text{net}, 3\text{-sun}) $-free (see Figure \ref{Fig:graphs}). 
\end{theorem}
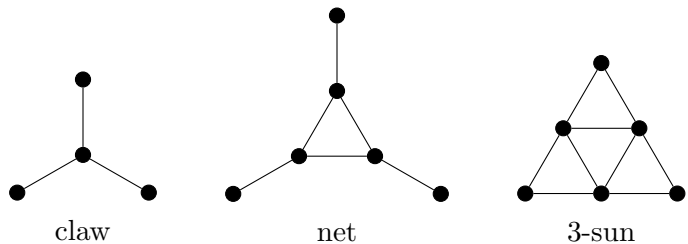
\begin{figure}[t]
\centering
\begin{tikzpicture}
\draw (0,0) node[v](1){};
\draw (0,1) node[v](2){};
\draw (-0.865,-0.5) node[v](3){};
\draw ( 0.865,-0.5) node[v](4){};
\draw (2)--(1)--(3);
\draw (1)--(4);
\draw (0,-1) node(){claw};
\end{tikzpicture}
\qquad
\begin{tikzpicture}
\draw (0,0) node[v](1){};
\draw (1,0) node[v](2){};
\draw (0.5,0.865) node[v](3){};
\draw (-0.865,-0.5) node[v](4){};
\draw ( 1.865,-0.5) node[v](5){};
\draw (0.5,1.865) node[v](6){};
\draw (4)--(1)--(3)--(6);
\draw (1)--(2)--(3);
\draw (2)--(5);
\draw (0.5,-1) node(){net};
\end{tikzpicture}
\qquad
\begin{tikzpicture}
\draw (0,0) node[v](1){};
\draw (1,0) node[v](2){};
\draw (2,0) node[v](3){};
\draw (0.5,0.865) node[v](4){};
\draw (1.5,0.865) node[v](5){};
\draw (1,1.73) node[v](6){};
\draw (6)--(4)--(1)--(2)--(3)--(5)--(6);
\draw (4)--(2)--(5)--(4);
\draw (1,-0.5) node(){$ 3 $-sun};
\end{tikzpicture}
\caption{Obstructions for unit interval graphs} \label{Fig:graphs}
\end{figure}

A graph in $ \UBG(S^{1}) $ is called a \textbf{unit circular-arc graph}, which is one of a natural generalization of unit interval graphs and this class is also well investigated (see \cite{lin2009characterizations-dm}, for example). 
Tucker \cite[Theorem 3.1 and Theorem 4.3]{tucker1974structure-dm} characterized unit circular-arc graphs in terms of forbidden induced subgraphs. 

A graph in $ \UBG(\mathbb{R}^{2}) $ is called a \textbf{unit disk graph}.
Breu and Kirkpatrick \cite{breu1998unit-cg} showed that the recognition of unit disk graphs is NP-hard. 
Therefore it seems very difficult to characterize unit disk graphs in terms of forbidden induced subgraphs. 
Recently, Atminas and Zamaraev \cite{atminas2018forbidden-dcg} discovered infinitely many minimal non-unit disk graphs. 

A graph in $ \UBG(\mathbb{R}^{2}, ||\ast||_{\infty}) = \UBG(\mathbb{R}^{2}, ||\ast||_{1}) $ is called a \textbf{unit square graph}. 
Breu \cite[Corollary 3.46.2]{breu1996algorithmic} proved that the recognition of unit square graphs is also NP-hard. 
Neuen \cite{neuen2016graph-aesoa} proved that the graph isomorphism problem for unit square graphs can be solved in polynomial time and investigated a lot of properties of unit square graphs. 
For instance, Neuen showed that every unit square graph is $ (K_{1,5}, K_{2,3}, \overline{3K_{2}}) $-free (see \cite[Lemma 3.3]{neuen2016graph-aesoa}).

From the above we obtain Table \ref{Tab:distinguishing table}. 
In particular, the collections of unit ball graphs of these four geodesic spaces are different. 
It is expected that the collection of unit ball graphs expresses some geometric properties of a geodesic space. 
\begin{table}[t]
\centering
\begin{tabular}{|c|c||c|c|c|c|c|}
\hline
class & space & $ C_{4} $ & claw & net & $ 3 $-sun & $ K_{1,5} $ \\
\hline 
unit interval \rule{0pt}{15pt} & $ \mathbb{R} $ & \XSolidBrush & \XSolidBrush & \XSolidBrush & \XSolidBrush & \XSolidBrush \\
unit circular-arc & $ S^{1} $ & \Checkmark & \XSolidBrush & \XSolidBrush & \XSolidBrush & \XSolidBrush \\
unit disk & $ \mathbb{R}^{2} $ & \Checkmark & \Checkmark & \Checkmark & \Checkmark & \Checkmark \\
unit square & $ (\mathbb{R}^{2}, ||\ast||_{\infty}) $ & \Checkmark & \Checkmark & \Checkmark & \Checkmark & \XSolidBrush \\
\hline
\end{tabular}
\caption{Geodesic spaces and their unit ball graphs} 
\label{Tab:distinguishing table}
\end{table}

Two metric spaces $ X $ and $ Y $ are said to be \textbf{similar} if there exists a bijection $ f \colon X \to Y $ and $ r > 0 $ such that $ d_{Y}(f(a), f(b)) = rd_{X}(a,b) $ for any $ a, b \in X $. 
When $ r = 1 $ we say that $ X $ and $ Y $ are \textbf{isometric}. 
If $ X $ and $ Y $ are similar, then clearly $ \UBG(X) = \UBG(Y) $. 

However, we can find easily non-similar geodesic spaces whose unit ball graphs coincide. 
For example, $ \UBG([0,1]) = \UBG(\mathbb{R}) $. 
More generally, when $ X $ is a convex subset of $ \mathbb{R}^{n} $ with non-empty interior, we have that $ \UBG(X) = \UBG(\mathbb{R}^{n}) $. 
In Section \ref{Sec:basic properties}, we will give a sufficient condition for the coincidence of the collections of unit ball graphs (Corollary \ref{sufficient condition for coincidence}).

The main contribution of this article is to characterize geodesic spaces whose unit ball graphs are chordal and geodesic spaces whose unit ball graphs are $ (\text{claw}, \text{net}) $-free. 
In order to state the results, we define $ \mathbb{R} $-trees and tripods. 
There are several equivalent definitions for $ \mathbb{R} $-trees. 
Here we give one of them (see Section \ref{Sec:R-trees}, for other conditions and details). 

\begin{definition}
A geodesic space is said to be an \textbf{$ \mathbb{R} $-tree} (or a \textbf{real tree}) if it is uniquely arc-connected, that is, every pair of points in it is joined by a unique arc. 
\end{definition}
\begin{definition}
A subset $ Y $ of a geodesic space $ X $ is said to be a \textbf{tripod} (see Figure \ref{Fig:tripod}) if there exist four distinct points $ x_{1}, x_{2}, x_{3}, y \in Y $ and geodesic segments $ [x_{i}, y]  \ (i=1,2,3) $ such that 
\begin{align*}
Y = [x_{1}, y] \cup  [x_{2}, y] \cup  [x_{3}, y] 
\text{ and } 
[x_{i}, y] \cap [x_{j}, y] = \{y\} \text{ if } i \neq j. 
\end{align*}
\end{definition}
\begin{figure}[t]
\centering
\begin{tikzpicture}
\draw (0,0) coordinate(1) node[above]{$ y $};
\draw (1,0.6) coordinate(2) node[right]{$ x_{1} $};
\draw (-1,0.3) coordinate(3) node[left]{$ x_{2} $};
\draw ( 0.2,-0.6) coordinate(4) node[below]{$ x_{3} $};
\draw (2)--(1)--(3);
\draw (1)--(4);
\end{tikzpicture}
\caption{A tripod} \label{Fig:tripod}
\end{figure}
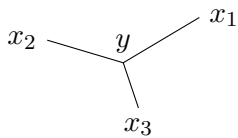

The main results are as follows. 
\begin{theorem}\label{main R-tree}
Let $ (X,d) $ be a geodesic space. 
Then the following are equivalent: 
\begin{enumerate}[(1)]
\item\label{main R-tree 1} Every unit ball graph on $ X $ is strongly chordal. 
\item\label{main R-tree 2} Every unit ball graph on $ X $ is chordal. 
\item\label{main R-tree 3} $ X $ is an $ \mathbb{R} $-tree. 
\end{enumerate}
\end{theorem}

\begin{remark}
Gavril \cite[Theorem 3]{gavril1974intersection-joctsb} showed that a graph is chordal if and only if it is a subtree graph, that is, the intersection graph of subtrees of a tree. 
If we choose subtrees in the following way, then the intersection graph is strongly chordal. 
Let $ T $ be a tree and $ X_{T} $ the $ \mathbb{R} $-tree obtained by replacing the edges of $ T $ with a copy of the unit interval $ [0,1] $. 
Fix a positive integer $ r $ and take some closed balls of radius $ r $ with center at the point corresponding to a vertex of $ T $. 
The intersection graph of these closed balls is strongly chordal by Theorem \ref{main R-tree}. 
For each of such closed balls, the points in it corresponding to vertices of $ T $ induce a subtree of $ T $. 
Then the intersection graph of these subtrees of $ T $ is also strongly chordal. 
It is not clear whether any strongly chordal graph is obtained by such a way or not. 
\end{remark}

\begin{theorem}\label{main 1-manifold}
Let $ (X,d) $ be a geodesic space. 
Then the following are equivalent: 
\begin{enumerate}[(1)]
\item\label{main 1-manifold 1} Every unit ball graph on $ X $ is (claw, net)-free. 
\item\label{main 1-manifold 2} $ X $ has no tripod. 
\item\label{main 1-manifold 3} $ X $ is homeomorphic to a manifold of dimension at most $ 1 $, that is, $ X $ is similar to $ S^{1} $ or isometric to an interval, that is, a convex subset of $ \mathbb{R} $. 
\end{enumerate}
\end{theorem}

\begin{remark} \label{Hamiltonian path}
According to \cite{damaschke1990hamiltonian-hereditary-u, duffus1981forbidden-ttaaog} (See also \cite[Theorem 7.1.7]{brandstadt1999graph}), every connected (claw, net)-free graph has a Hamiltonian path. 
Hence we can deduce that a graph is (claw, net)-free if and only if it is a Hamiltonian-hereditary graph, that is, every connected induced subgraph of it has a Hamiltonian path. 

Clearly, unit ball graphs on intervals and $ S^{1} $ are Hamiltonian-hereditary graphs. 
Hence $ (\ref{main 1-manifold 3}) \Rightarrow (\ref{main 1-manifold 1}) $ of Theorem \ref{main 1-manifold} is true. 
Moreover, Theorem \ref{main 1-manifold} asserts that geodesic spaces whose unit ball graphs are Hamiltonian-hereditary graphs are exactly intervals and $ S^{1} $. 
\end{remark}

Theorem \ref{main R-tree}, \ref{main 1-manifold}, and Table \ref{Tab:distinguishing table} 
lead to the following corollaries. 

\begin{corollary}\label{main interval}
Let $ (X,d) $ be a geodesic space. 
Then the following are equivalent: 
\begin{enumerate}[(1)]
\item $ \UBG(X) = \UBG(\mathbb{R}) $. 
\item $ X $ is isometric to an interval which is not the single-point space. 
\end{enumerate}
\end{corollary}

\begin{corollary}\label{main circle}
Let $ (X,d) $ be a geodesic space. 
Then the following are equivalent: 
\begin{enumerate}[(1)]
\item $ \UBG(X) = \UBG(S^{1}) $. 
\item $ X $ is similar to $ S^{1} $. 
\end{enumerate}
\end{corollary}

The organization of this article is as follows. 
In Section \ref{Sec:basic properties}, we study basic properties of unit ball graphs. 
In Section \ref{Sec:R-trees}, we review the theory of $ \mathbb{R} $-trees and prove Theorem \ref{main R-tree}. 
In Section \ref{Sec:1-manifold}, we will prove Theorem \ref{main 1-manifold}. 

\section{Basic properties}\label{Sec:basic properties}
First of all, we begin with the following proposition, which is easy to prove. 
\begin{proposition}
The class of unit ball graphs on a metric space $ (X,d) $ is hereditary. 
Namely, if $ G \in \UBG(X) $, then every induced subgraph of $ G $ belongs to $ \UBG(X) $. 
\end{proposition}
Therefore every class $ \UBG(X) $ has a characterization in terms of forbidden induced subgraphs. 
However, as mentioned in the introduction, it is difficult to characterize $ \UBG(X) $ in general. 
Next, we treat the most trivial case, that is, the single-point space $ \{\ast\} $. 
\begin{proposition}\label{single-point}
The following statements hold true: 
\begin{enumerate}[(1)]
\item \label{single-point 1} $ \UBG(\{\ast\}) $ consists of complete graphs, or equivalently $ 2K_{1} $-free graphs. 
\item \label{single-point 2} Let $ (X,d) $ be a metric space. 
Then $ \UBG(X) = \UBG(\{\ast\}) $ if and only if $ X = \{\ast\} $. 
\end{enumerate}
\end{proposition}
\begin{proof}
The assertion (\ref{single-point 1}) is trivial. 
To show (\ref{single-point 2}), take two distinct points $ x, y \in X $. 
Then the unit ball graph on $ \{x, y\} $ with threshold $ d(x,y)/3 $ is $ 2K_{1} $. 
Thus the assertion holds. 
\end{proof}

By definition, a unit ball graph is the intersection graph of finitely many \emph{closed} balls of the same size. 
Next we show that we may use \emph{open} balls instead of closed ones. 
\begin{proposition}
A simple graph $ G=(V_{G},E_{G}) $ is a unit ball graph on a metric space $ (X,d) $ if and only if  there exist $ \delta > 0 $ and a map $ \rho \colon V_{G} \to X $ such that $ \{u,v\} \in E_{G} $ if and only if $ d(\rho(u), \rho(v)) < \delta $. 
\end{proposition}
\begin{proof}
Assume that $ G $ is a unit ball graph on $ X $. 
Then, by definition, there exist $ \delta > 0 $ and a map $ \rho \colon V_{G} \to X $ such that $ \{u,v\} \in E_{G} $ if and only if $ d(\rho(u), \rho(v)) \leq \delta $. 
Let $ \delta^{\prime} = \min\Set{d(\rho(u), \rho(v)) | \{u,v\} \not\in E_{G}} $. 
Then we have $ \{u,v\} \in E_{G} $ if and only if $ d(\rho(u), \rho(v)) < \delta^{\prime} $. 
The converse can be proven in a similar way.
\end{proof}

Next, we give a sufficient condition for inclusion of the classes of unit ball graphs. 
\begin{proposition}\label{sufficient condition for inclusion}
Let $ (X,d_{X}) $ and $ (Y,d_{Y}) $ be metric spaces. 
Suppose that every finite subset in $ X $ is similarly embedded into $ Y $. 
Namely, assume that, for any finite subset $ S $, there exist $ r > 0 $ and a map $ f \colon S \to Y $ such that $ d_{Y}(f(a),f(b)) = rd_{X}(a,b) $ for any $ a,b \in S $. 
Then $ \UBG(X) \subseteq \UBG(Y) $. 
\end{proposition}
\begin{proof}
Let $ G \in \UBG(X) $ with a realization $ \rho $ and a threshold $ \delta $. 
Since $ \rho(V_{G}) $ is finite, by the assumption, there exist $ r > 0 $ and $ f \colon \rho(V_{G}) \to Y $ such that $ d_{Y}(f(\rho(u)), f(\rho(v))) = rd_{X}(\rho(u),\rho(v)) $ for any $ u,v \in V_{G} $. 
Hence we conclude that $ G \in \UBG(Y) $ with a realization $ f \circ \rho $ and a threshold $ \delta/r $. 
\end{proof}
\begin{corollary}\label{sufficient condition for coincidence}
Let $ (X,d_{X}) $ and $ (Y,d_{Y}) $ be metric spaces. 
Suppose that every finite subset in $ X $ is similarly embedded into $ Y $, and vice versa.
Then $ \UBG(X) = \UBG(Y) $. 
\end{corollary}

From Proposition \ref{sufficient condition for inclusion}, we have that $ \UBG(\mathbb{R}^{m}) \subseteq \UBG(\mathbb{R}^{n}) $ whenever $ m \leq n $. 
Therefore, intuitively, higher dimensional spaces could have more unit ball graphs. 
However, the converse is not true in general as follows. 
\begin{proposition}
There exists a geodesic space $ X $ such that its Lebesgue covering dimension is $ 1 $ and $ \UBG(X) $ consists of all graphs. 
\end{proposition}
\begin{proof}
Let $ G $ be a connected graph and $ X_{G} $ the geodesic space obtained by replacing the edges of $ G $ with a copy of the unit interval $ [0,1] $ (so $ X_{G} $ is the underlying space of a $ 1 $-dimensional simplicial complex). 
Clearly, $ G \in \UBG(X_{G}) $. 
Choose a vertex of $ G $ and connect a geodesic segment of length $ 1 $ to the corresponding point of $ X_{G} $. 
Let $ X $ be a geodesic space obtained by gluing the other endpoints with respect to each connected graph and its countably many copies. 
Then every graph belongs to $ \UBG(X) $. 
Obviously, the Lebesgue covering dimension of $ X $ is $ 1 $. 
\end{proof}

It is not clear whether the converse of Corollary \ref{sufficient condition for coincidence} is true or not. 
However, for geodesic spaces $ \mathbb{R}, S^{1} $, and $ \{\ast\} $, it is true by Corollary \ref{main interval}, \ref{main circle}, and Proposition \ref{single-point} (\ref{single-point 2}). 

\section{Chordal graphs and $ \mathbb{R} $-trees }\label{Sec:R-trees}
In this section, we will give a proof of Theorem \ref{main R-tree}. 
\begin{theorem}[Restate of Theorem \ref{main R-tree}]
Let $ (X,d) $ be a geodesic space. 
Then the following are equivalent: 
\begin{enumerate}[(1)]
\item Every unit ball graph on $ X $ is strongly chordal. 
\item Every unit ball graph on $ X $ is chordal. 
\item $ X $ is an $ \mathbb{R} $-tree. 
\end{enumerate}
\end{theorem}

Note that the implication $ (\ref{main R-tree 1}) \Rightarrow (\ref{main R-tree 2}) $ follows by definition. 
As mentioned before, an $ \mathbb{R} $-tree is a geodesic space in which every pair of points is joined by a unique arc, that is, the image of a \emph{topological} embedding of a closed interval. 
Note that every arc in an $ \mathbb{R} $-tree is a (unique) geodesic segment and hence an $ \mathbb{R} $-tree is uniquely geodesic. 
Obviously, the real line $ \mathbb{R} $ and intervals are $ \mathbb{R} $-trees. 
The underlying space of a $ 1 $-dimensional connected acyclic simplicial complex is also an $ \mathbb{R} $-tree. 

\begin{proposition}[See {\cite[Proposition 2.3]{chiswell2001introduction}} and {\cite{bridson1999metric}}, for example]\label{R-tree equivalent conditions}
Let $ (X,d) $ be a geodesic space. 
Then the following conditions are equivalent. 
\begin{enumerate}[(1)]
\item $ X $ is an $ \mathbb{R} $-tree 
\item $ X $ has no subspace homeomorphic to $ S^{1} $. 
\item For any $ x,y,z \in X $, whenever $ [x,y] \cap [y,z] = \{y\} $, the union $ [x,y] \cup [y,z] $ is a geodesic segment joining $ x $ and $ z $. 
\item\label{R-tree equivalent conditions 4} $ X $ is a Gromov $ 0 $-hyperbolic space, that is, for any geodesic segments $ [x,y], [x,z], [y,z] $, we have $ [x,z] \subseteq [x,y] \cup [y,z] $. 
\end{enumerate}
\end{proposition}

\subsection{Proof of Theorem \ref{main R-tree} (\ref{main R-tree 2}) $ \Rightarrow $ (\ref{main R-tree 3})}
The following two propositions are required. 
The first one is very famous. 
\begin{proposition}\label{separation}
Suppose that $ A $ and $ B $ be disjoint compact subsets of a metric space $ (X, d) $. 
Then
\begin{align*}
d(A,B) \coloneqq \inf_{\substack{a \in A \\ b \in B}} d(a,b) > 0. 
\end{align*}
\end{proposition}

\begin{proposition}\label{Connectedness}
Let $ S $ be a connected subset of a topological space and $ \mathcal{U} $ a finite open covering of $ S $ in which each member intersects with $ S $. 
Then the intersection graph of $ \mathcal{U} $ is connected. 
\end{proposition}
\begin{proof}
Assume that the intersection graph of $ \mathcal{U} $ is disconnected. 
Then there exist non-empty subsets $ \mathcal{U}_{1} $ and $ \mathcal{U}_{2} $ of $ \mathcal{U} $ such that $ \mathcal{U} = \mathcal{U}_{1} \cup \mathcal{U}_{2}, \ \mathcal{U}_{1} \cap \mathcal{U}_{2} = \varnothing $, and $ U_{1} \cap U_{2} = \varnothing $ for any $ U_{i} \in \mathcal{U}_{i} (i =1,2) $. 
For each $ i \in \{1,2\} $, let $ O_{i} \coloneqq \bigcup_{U \in \mathcal{U}_{i}} U$. 
Then $ S \subseteq O_{1} \cup O_{2} $ and $ S \cap O_{i} \neq \varnothing $ for each $ i \in \{1,2\} $. 
Therefore $ S $ is disconnected. 
Thus the assertion holds. 
\end{proof}

\begin{proof}[Proof of Theorem \ref{main R-tree} (\ref{main R-tree 2}) $ \Rightarrow $ (\ref{main R-tree 3})]
We assume that $ X $ is not an $ \mathbb{R} $-tree and show that there exists $ G \in \UBG(X) $ such that $ G $ is not chordal. 
By Proposition \ref{R-tree equivalent conditions}, there exists a topological embedding $ \phi \colon S^{1} \to X $. 
We consider $ S^{1} $ as $ \mathbb{R}/4\mathbb{Z} $ and put $ p_{i} \coloneqq \phi(i), 
\text{ and } S_{i} \coloneqq \phi([i,i+1]) \text{ for } i \in \{1,2,3,4\} \subseteq \mathbb{R}/4\mathbb{Z} $ (see Figure \ref{Fig:image of S^1}). 
By Proposition \ref{separation}, there exists $ r > 0 $ such that $ d(S_{1}, S_{3}) > 2r $ and $ d(S_{2}, S_{4}) > 2r $. 
Let $ U_{r}(x) $ denote the open ball of radius $ r $ with center $ x $. 
Note that $ U_{r}(x) \cap U_{r}(y) = \varnothing $ whenever $ (x,y) \in S_{1} \times S_{3} $ or $ (x,y) \in S_{2} \times S_{4} $. 
For every $ i $, $ \Set{U_{r}(x)}_{x \in S_{i}} $ is an open covering of $ S_{i} $. 
Since $ S_{i} $ is compact, we have a finite subset $ V_{i} \subseteq S_{i} $ such that $ \Set{U_{r}(x)}_{x \in V_{i}} $ is a finite open covering of $ S_{i} $. 
Adding two points $ p_{i}, p_{i+1} $ if necessary, we may suppose that $ p_{i}, p_{i+1} \in V_{i} $. 
Let $ G $ be the intersection graph of $ \Set{U_{r}(x)}_{x \in V_{1} \cup \dots \cup V_{4}} $ and we will identify the vertices of $ G $ with the corresponding points. 
We will show that $ G $ is not chordal. 
By Proposition \ref{Connectedness}, each induced subgraph $ G[V_{i}] $ is connected. 
Take a shortest path $ \pi_{i} $ from $ p_{i} $ to $ p_{i+1} $ in $ G[V_{i}] $. 
Since $ p_{i} \in S_{i-1} $ and $ p_{i+1} \in S_{i+1} $, they are not adjacent. 
Therefore the length of the path $ \pi_{i} $ is at least two and every intermediate vertex of $ \pi_{i} $ is an interior point of $ S_{i} $. 
Connecting the paths $ \pi_{1}, \dots, \pi_{4} $, we have a cycle $ C = (V_{C}, E_{C})$ satisfying the following conditions. 
\begin{enumerate}[(i)]
\item\label{cycle condition i} $ V_{C} \cap V_{i} \neq \varnothing $ and $ C[V_{C} \cap V_{i}] $ is a chordless path for each $ i $. 
\item\label{cycle condition ii} $ C $ has a vertex corresponding to an interior point of $ S_{i} $ for each $ i $. 
\end{enumerate}
Suppose that $ C_{0} $ is a minimal cycle satisfying these conditions. 
By the condition (\ref{cycle condition ii}), the length of $ C_{0} $ is at least four. 
Let $ i \in \{1,2,3,4\} $. 
From (\ref{cycle condition i}), there exists no chord between two vertices in $ V_{C_{0}} \cap V_{i} $. 
By minimality of $ C_{0} $, there exists no chord connecting $ V_{C_{0}} \cap V_{i} $ and $ V_{C_{0}} \cap V_{i+1} $. 
By choice of $ r $, there exists no chord joining $ V_{C_{0}} \cap V_{i} $ and $ V_{C_{0}} \cap V_{i+2} $. 
Thus $ C_{0} $ is a chordless cycle of length at least four and hence $ G $ is a non-chordal unit ball graph on $ X $. 
\end{proof}

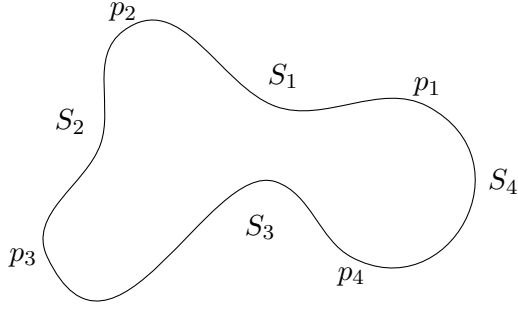
\begin{figure}[t]
\centering
\begin{tikzpicture}
\draw (0,0) coordinate[label=left:{$ p_{3} $}](1);
\draw (3,1) coordinate(2);
\draw (4,0) coordinate[label=below:{$ p_{4} $}](3);
\draw (5,2) coordinate[label=above:{$ p_{1} $}](4);
\draw (3,2) coordinate(5);
\draw (1,3) coordinate[label=above:{$ p_{2} $}](6);
\draw (0.7,1.5) coordinate(7);
\path[draw,use Hobby shortcut,closed=true] (1)..(2)..(3)..(4)..(5)..(6)..(7);
\draw (3.1,2.4) node{$ S_{1} $}; 
\draw (0.3,1.8) node{$ S_{2} $}; 
\draw (2.8,0.4) node{$ S_{3} $}; 
\draw (6,1) node{$ S_{4} $}; 
\end{tikzpicture}
\caption{The image of $ S^{1} $} \label{Fig:image of S^1}
\end{figure}

\subsection{Proof of Theorem \ref{main R-tree} (\ref{main R-tree 3}) $ \Rightarrow $ (\ref{main R-tree 1})}
\begin{proposition}\label{recombination}
Let $ (X,d) $ be a geodesic space. 
Then the following conditions hold. 
\begin{enumerate}[(1)]
\item\label{recombination 1} 
Let $ x_{1}, \dots, x_{n}, y_{1}, \dots, y_{n} \in X $. 
Suppose that $ \bigcap_{i=1}^{n}[x_{i}, y_{i}] \neq \varnothing $. 
Then $ \sum_{i=1}^{n}d(x_{i},y_{i}) \geq \sum_{i=1}^{n}d(x_{i},y_{\sigma(i)}) $ for any permutation $ \sigma $. 
\item \label{recombination 2}
Let $ G \in \UBG(X) $. 
Suppose that $ \{x_{i}, y_{i}\} \in E_{G} $ for each $ i \in \{1, \dots, n\} $ and there exists a permutation $ \sigma $ such that $ \{x_{i}, y_{\sigma(i)}\} \not\in E_{G} $ for any $ i $. 
Then $ \bigcap_{i=1}^{n}[x_{i}, y_{i}] = \varnothing $. 
\end{enumerate}
\end{proposition}
\begin{proof}
(\ref{recombination 1})
Take a point $ p \in \bigcap_{i=1}^{n}[x_{i}, y_{i}] $. 
Then we have
\begin{align*}
\sum_{i=1}^{n}d(x_{i},y_{i}) 
&= \sum_{i=1}^{n}\left(d(x_{i},p) + d(p,y_{i})\right) 
= \sum_{i=1}^{n}d(x_{i},p) + \sum_{i=1}^{n}d(p,y_{i}) \\
&= \sum_{i=1}^{n}d(x_{i},p) + \sum_{i=1}^{n}d(p,y_{\sigma(i)}) 
= \sum_{i=1}^{n}\left(d(x_{i},p) + d(p,y_{\sigma(i)})\right)
\geq \sum_{i=1}^{n}d(x_{i},y_{\sigma(i)}), 
\end{align*}
where we apply the triangle inequality. 

(\ref{recombination 2}) 
By the assumption, there exists a threshold $ \delta > 0 $ such that $ d(x_{i}, y_{i}) \leq \delta $ and $ d(x_{i},y_{\sigma(i)}) > \delta $ for any $ i \in \{1, \dots, n\} $. 
Assume that $ \bigcap_{i=1}^{n}[x_{i}, y_{i}] \neq \varnothing $. 
Then by the assertion (\ref{recombination 1}) we have 
\begin{align*}
n\delta 
\geq \sum_{i=1}^{n}d(x_{i},y_{i}) 
\geq \sum_{i=1}^{n}d(x_{i},y_{\sigma(i)})
> n\delta. 
\end{align*}
This contradiction proves the assertion. 
\end{proof}

\begin{proposition}[Alperin-Bass {\cite[Proposition 2.17]{alperin1987length-cgtat}}, Chiswell {\cite[Lemma 1.9(b)]{chiswell2001introduction}}]\label{bridge}
Let $ A $ and $ B $ be non-empty closed connected subsets of an $ \mathbb{R} $-tree such that $ A \cap B = \varnothing $. 
Then there exist unique points $ a \in A $ and $ b \in B $ such that $ [a,b] \cap A = \{a\} $ and $ [a,b] \cap B = \{b\} $. 
\end{proposition}
\begin{definition}
We call $ [a,b] $ in Proposition \ref{bridge} the \textbf{bridge} between $ A $ and $ B $. 
Let $ Y $ be a non-empty closed connected subset of an $ \mathbb{R} $-tree and $ x $ a point. 
Call the other endpoint $ y $ of the bridge $ [x,y] $ between $ \{x\} $ and $ Y $ the \textbf{closest point} in $ Y $ to $ x $ when $ x \not\in Y $. 
If $ x \in Y $, define the closest point as $ x $ itself. 
\end{definition}

\begin{lemma}\label{UBG cycle}
Suppose that the $ n $-cycle $ C_{n} $ is a unit ball graph of a geodesic space $ X $. 
Let $ \{x_{i}\}_{i \in \mathbb{Z}/n\mathbb{Z}} $ be the vertex set of $ C_{n} $ with $ \{x_{i}, x_{i+1}\} \in E_{C_{n}} \ (i \in \mathbb{Z}/n\mathbb{Z}) $. 
Suppose that $ \{x_{i}, x_{i+1}\}, \{x_{j}, x_{j+1}\} $ are non-adjacent edges. 
Then $ [x_{i},x_{i+1}] \cap [x_{j},x_{j+1}] = \varnothing $. 
\end{lemma}
\begin{proof}
Since $ \{x_{i}, x_{i+1}\}, \{x_{j}, x_{j+1}\} \in E_{C_{n}} $ and $ \{x_{i}, x_{j}\}, \{x_{i+1}, x_{j+1}\} \not\in E_{C_{n}} $, we have immediately $ [x_{i},x_{i+1}] \cap [x_{j},x_{j+1}] = \varnothing $ by Proposition \ref{recombination} (\ref{recombination 2}). 
\end{proof}

\begin{lemma}\label{UBG sun}
Suppose that a geodesic space $ X $ admits an $ n $-sun as a unit ball graph. 
Let $ \{x_{i}\} _{i \in \mathbb{Z}/2n\mathbb{Z}}$ be the vertex set such that even-indexed vertices induce a clique, odd-indexed vertices are independent, and two consecutive vertices form an edge. 
Then the following assertions hold. 
\begin{enumerate}[(1)]
\item\label{UBG sun 1}
Suppose that $ \{x_{i}, x_{i+1}\} $ and $ \{x_{j}, x_{j+1}\} $ are non-adjacent edges with $ j \neq i \pm 2 $. 
Then $ [x_{i}, x_{i+1}] \cap [x_{j}, x_{j+1}] = \varnothing $. 
\item\label{UBG sun 2}
When $ \{x_{i},x_{i+1}\}, \{x_{j},x_{j+1}\} $, and $ \{x_{k},x_{k+1}\} $ are non-adjacent edges, we have that $ [x_{i},x_{i+1}] \cap [x_{j},x_{j+1}] \cap [x_{k},x_{k+1}] = \varnothing $. 
\end{enumerate}
\end{lemma}
\begin{proof}
(\ref{UBG sun 1}) 
Recall that an even-indexed vertex and an odd-indexed vertex in the $ n $-sun are adjacent if and only if their indices are consecutive. 
If $ i \not\equiv j \pmod{2} $, then neither $ \{x_{i}, x_{j}\} $ nor $ \{x_{i+1}, x_{j+1} \} $ is an edge. 
Otherwise, neither $ \{x_{i}, x_{j+1}\} $ nor $ \{x_{j}, x_{i+1}\} $ is an edge since $ j \neq i \pm 2 $. 
In any cases, we may conclude that the assertion is true by Proposition \ref{recombination} (\ref{recombination 2}). 

(\ref{UBG sun 2}) 
For each $ s \in \{i,j,k\} $, let $ p_{s} $ and $ q_{s} $ be the odd-indexed and even-indexed vertices of $ \{x_{s}, x_{s+1}\} $. 
We will show that neither $ \{p_{i}, q_{j}\}, \{p_{j}, q_{k}\} $, nor $ \{p_{k}, q_{i}\} $ is an edge or neither $ \{p_{i}, q_{k}\}, \{p_{k}, q_{j}\} $, nor $ \{p_{j}, q_{i}\} $ is an edge.  
Suppose that the former does not hold. 
Without loss of generality we may assume that $ \{p_{i},q_{j}\} $ is an edge. 
Then the indices of the vertices $ q_{i}, p_{i}, q_{j}, p_{j} $ are consecutive in this order. 
Since the cardinality of the vertex set is at least $ 6 $, we have $ \{p_{j}, q_{i}\} $ is not an edge. 
Moreover, neither $ \{p_{i}, q_{k}\} $ nor $ \{p_{k}, q_{j}\} $ is an edge. 
Thus the latter condition holds. 
Using Proposition \ref{recombination} (\ref{recombination 2}), we have proved the assertion. 
\end{proof}

\begin{proof}[Proof of Theorem \ref{main R-tree} (\ref{main R-tree 3}) $ \Rightarrow $ (\ref{main R-tree 1})]
We will show that every unit ball graph on an $ \mathbb{R} $-tree is chordal and sun-free. 
Assume that there exists $ n \geq 4 $ such that $ C_{n} \in \UBG(X) $. 
Let $ \{x_{i}\}_{i \in \mathbb{Z}/n\mathbb{Z}} $ be the vertices of $ C_{n} $ with $ \{x_{i}, x_{i+1}\} \in E_{C_{n}} $. 
Then $ A \coloneqq [x_{1}, x_{2}] $ and $ B \coloneqq \bigcup_{i=3}^{n-1}[x_{i},x_{i+1}] $ are disjoint closed connected subsets by Lemma \ref{UBG cycle}. 
By Proposition \ref{bridge}, both of the segments $ [x_{1},x_{n}] $ and $ [x_{2},x_{3}] $ contains the bridge between $ A $ and $ B $. 
In particular, we have $ [x_{1},x_{n}] \cap [x_{2},x_{3}] \neq \varnothing $, which contradicts to Lemma \ref{UBG cycle}. 
Thus we conclude that every unit ball graph on $ X $ is chordal. 

Next, suppose that an $ n $-sun is a unit ball graph on $ X $. 
Let $ \{x_{i}\}_{i \in \mathbb{Z}/2n\mathbb{Z}} $ be the vertex set of the $ n $-sun with the conditions mentioned in Lemma \ref{UBG sun}. 
Let $ A \coloneqq [x_{1}, x_{2}] $ and $ B \coloneqq \bigcup_{i=4}^{2n-2}[x_{i},x_{i+1}] $. 
By Lemma \ref{UBG sun} (\ref{UBG sun 1}), we have $ A $ and $ B $ are disjoint closed connected subsets of $ X $. 
Let $ [a,b] $ be the bridge between $ A $ and $ B $ and set $ Y \coloneqq A \cup [a,b] \cup B $. 
Let $ y_{3} $ and $ y_{2n} $ be the closest points in $ Y $ to $ x_{3} $ and $ x_{2n} $. 

We will show that $ y_{3}, y_{2n} \in A \cup B $. 
Assume that $ y_{3} \not\in A \cup B $. 
Then we have $ y_{3} \in [a,b] \subseteq [x_{1}, x_{2n-1}] \subseteq [x_{2n}, x_{1}] \cup [x_{2n-1}, x_{2n}] $ by Proposition \ref{bridge} and \ref{R-tree equivalent conditions} (\ref{R-tree equivalent conditions 4}). 
Furthermore, by Proposition \ref{bridge} again, we have $ y_{3} \in [x_{3}, x_{4}] \cap [x_{2}, x_{3}] $. 
Therefore $ y_{3} \in ([x_{2n},x_{1}] \cap [x_{3},x_{4}]) \cup ([x_{2n-1},x_{2n}] \cap [x_{2},x_{3}]) $. 
However, by Lemma \ref{UBG sun} (\ref{UBG sun 1}), we have $ [x_{2n},x_{1}] \cap [x_{3},x_{4}] = [x_{2n-1},x_{2n}] \cap [x_{2},x_{3}] = \varnothing $, a contradiction. 
Hence $ y_{3} \in A \cup B $. 
We can show that $ y_{2n} \in A \cup B $ in a similar way. 

Assume that $ y_{3} \in A $ and $ y_{2n} \in B $. 
Then $ [x_{3},x_{4}] \cap [x_{2n},x_{1}] \supseteq [a,b] $, which contradicts to Lemma \ref{UBG sun} (\ref{UBG sun 1}). 
The condition $ y_{3} \in B $ and $ y_{2n} \in A $ also leads to a contradiction. 
If $ y_{3}, y_{2n} \in A $, then $ [x_{3}, x_{4}] \cap [x_{2n},x_{2n-1}] \supseteq [a,b] $. 
Therefore $ A \cap [x_{3}, x_{4}] \cap [x_{2n},x_{2n-1}] \ni a $. 
This is a contradiction to Lemma \ref{UBG sun} (\ref{UBG sun 2}). 
Finally assume that $ y_{3},y_{2n} \in B $. 
Then we have $ B \cap [x_{3},x_{2}] \cap [x_{2n}, x_{1}] \ni b $, which is again a contradiction to Lemma \ref{UBG sun} (\ref{UBG sun 2}).
Therefore we conclude that every unit ball graph on $ X $ is sun-free.  
Thus the proof has been completed. 
\end{proof}

\section{$ (\text{claw},\text{net}) $-free graphs and $ 1 $-dimensional manifolds}\label{Sec:1-manifold}
In this section, we will prove Theorem \ref{main 1-manifold}. 
\begin{theorem}[Restate of Theorem \ref{main 1-manifold}]
Let $ (X,d) $ be a geodesic space. 
Then the following are equivalent: 
\begin{enumerate}[(1)]
\item Every unit ball graph on $ X $ is (claw, net)-free. 
\item $ X $ has no tripod. 
\item $ X $ is homeomorphic to a manifold of dimension at most $ 1 $, that is, $ X $ is similar to $ S^{1} $ or isometric to an interval, that is, a convex subset of $ \mathbb{R} $. 
\end{enumerate}
\end{theorem}
As mentioned Remark \ref{Hamiltonian path}, the implication $ (\ref{main 1-manifold 3}) \Rightarrow (\ref{main 1-manifold 1}) $ holds.

\subsection{Proof of Theorem \ref{main 1-manifold} $ (\ref{main 1-manifold 1}) \Rightarrow (\ref{main 1-manifold 2}) $}
\begin{lemma}\label{net-like graph}
Let $ G $ be a graph on vertex set $ \{a_{i}\}_{i=0}^{l} \cup \{b_{i}\}_{i=0}^{m} \cup \{c_{i}\}_{i=0}^{n} $ with positive integers $ l,m,n $ satisfying the following conditions. 
\begin{enumerate}[(i)]
\item\label{net-like graph 1} $ \{a_{i}\}_{i=0}^{l}, \{b_{i}\}_{i=0}^{m}, \{c_{i}\}_{i=0}^{n} $ induce chordless paths. 
\item\label{net-like graph 2} $ a_{0}, b_{0}, c_{0} $ are leaves, that is, they are vertices of degree $ 1 $. 
\item\label{net-like graph 3} $ \{a_{l}, b_{m}, c_{n}\} $ induces a triangle. 
\end{enumerate}
Then $ G $ has an induced subgraph isomorphic to a claw or a net. 
\end{lemma}
\begin{proof}
We proceed by induction of the number of the vertices of $ G $. 
First suppose that $ l=m=n=1 $. 
Then by (\ref{net-like graph 1}) and (\ref{net-like graph 3}) 
\begin{align*}
E_{G} \supseteq \Set{\{a_{0},a_{1}\}, \{b_{0}, b_{1}\}, \{c_{0}, c_{1}\}, \{a_{1},b_{1}\}, \{b_{1},c_{1}\}, \{c_{1},a_{1}\}} \eqqcolon S. 
\end{align*}
Assume that there exists an edge $ e \in E_{G}\setminus S $. 
Then one of $ a_{0}, b_{0}, $ and $ c_{0} $ is incident to $ e $, which is a contradiction to (\ref{net-like graph 2}). 
Therefore $ E_{G}=S $ and $ G $ itself is a net. 
Hence we may assume that $ l \geq 2 $ by symmetry. 

If the neighborhood of $ a_{1} $ is $ \{a_{0}, a_{2}\} $, then $ G\setminus \{a_{0}\} $ satisfies the assumptions. 
Therefore, by the induction hypothesis, $ G\setminus\{a_{0}\} $ and hence $ G $ have a claw or a net. 
Without loss of generality, we may assume that there exists the minimum integer $ i $ such that $ \{a_{1}, b_{i}\} $ is an edge of $ G $. 
By assumption (\ref{net-like graph 2}) we have $ 1 \leq i \leq m $. 

Assume that $ i < m $. 
If $ \{a_{1}, b_{i+1}\} $ is not an edge of $ G $, then the four vertices $ a_{1}, b_{i-1}, b_{i}, b_{i+1} $ form a claw by the minimality of $ i $ and the assumption (\ref{net-like graph 1}). 
Now suppose that $ \{a_{1}, b_{i+1}\} $ is an edge. 
Note that $ \{a_{1}, b_{i}, b_{i+1}\} $ induces a triangle. 
Take a shortest path from $ b_{i+1} $ to $ c_{0} $ in the induced subgraph on $ \{b_{i+1}, \dots, b_{m}, c_{n}, \dots, c_{0}\} $. 
This path together with two paths on $ \{a_{0},a_{1}\} $ and $ \{b_{0}, \dots, b_{i}\} $ induce a subgraph of $ G $ satisfying the assumptions. 
This subgraph is a proper subgraph of $ G $ since this does not contain the vertex $ a_{l} $. 
Making use of the induction hypothesis, we have that $ G $ has a claw or a net. 
Hence we may assume that $ i=m $. 

If $ \{a_{1}, c_{n}\} $ is an edge of $ G $, then the graph $ G \setminus \{a_{2}, \dots, a_{l} \} $ satisfies the assumptions with the triangle $ \{a_{1}, b_{m}, c_{n}\} $. 
Therefore we may assume that $ \{a_{1}, c_{n}\} $ is not an edge of $ G $.
If $ \{b_{m-1}, c_{n}\} $ is an edge of $ G $, then the paths on $ \{a_{0}, a_{1}, b_{m}\}, \{b_{0}, \dots, b_{m-1}\}, \{c_{0}, \dots, c_{n}\} $ induce a subgraph satisfying the assumptions with the triangle $ \{b_{m}, b_{m-1}, c_{n}\} $. 
Hence we may assume that $ \{b_{m-1}, c_{n}\} $ is not an edge of $ G $. 
Then the induced subgraph on $ \{b_{m}, a_{1}, b_{m-1}, c_{n}\} $ is a claw. 
\end{proof}

\begin{proof}[Proof of Theorem \ref{main 1-manifold} $ (\ref{main 1-manifold 1}) \Rightarrow (\ref{main 1-manifold 2}) $]
Assume that there exist four points $ x_{1},x_{2},x_{3},y \in X $ forming a tripod with center $ y $. 
By Proposition \ref{separation}, there exists $ r > 0 $ such that 
\begin{align*}
d(\{ x_{i} \}, [x_{i+1},y] \cup [x_{i+2},y]) > 2r \text{ for every } i \in \mathbb{Z}/3\mathbb{Z}. 
\end{align*}
Let $ \gamma $ be a geodesic from $ x_{1} $ to $ y $ and $ l $ the greatest integer less than or equal to $ d(x_{1},y) $. 
Define a sequence $ \{a_{i}\}_{i=0}^{l} $ by $ a_{i} \coloneqq \gamma(ir) $. 
Note that $ U_{r}(a_{i}) \cap U_{r}(a_{j}) \neq \varnothing $ if and only if $ |i-j| \leq 1 $. 
Define sequences $ \{b_{i}\}_{i=0}^{m} $ and $ \{c_{i}\}_{i=0}^{n} $ with respect to $ x_{2} $ and $ x_{3} $ in a similar way. 
By the choice of $ r $, we have $ U_{r}(a_{0}) \cap U_{r}(b_{i}) = \varnothing,  \ U_{r}(a_{0}) \cap U_{r}(c_{i}) = \varnothing $, and so on. 
Moreover we have $ U_{r}(a_{l}) \cap U_{r}(b_{m}) \cap U_{r}(c_{n}) \ni y $. 
Therefore the intersection graph of open balls of radius $ r $ with center points in $ \{a_{i}\}_{i=0}^{l} \cup \{b_{i}\}_{i=0}^{m} \cup \{c_{i}\}_{i=0}^{n} $ satisfies the assumptions of Lemma \ref{net-like graph} and hence it has a claw or a net as an induced subgraph. 
\end{proof}

\subsection{Proof of Theorem \ref{main 1-manifold} $ (\ref{main 1-manifold 2}) \Rightarrow (\ref{main 1-manifold 3}) $}
\begin{lemma}\label{no tripod lemma}
Let $ (X,d) $ be a geodesic space with no tripods. 
Then the following conditions hold. 
\begin{enumerate}[(1)]
\item\label{no tripod lemma 1} Suppose that $ x,y,z \in X $ satisfy $ [x,z] \cap [y,z] \supsetneq \{z\} $. 
Then $ x \in [y,z] $ or $ y \in [x,z] $. 
In particular, if $ X $ is uniquely geodesic, then $ [x,z] \subseteq [y,z] $ or $ [y,z] \subseteq [x,z] $. 
\item\label{no tripod lemma 2} Let $ x,y,z $ be distinct points with $ y \in [x,z] $. 
Then the geodesic segment between $ x $ and $ y $ is unique and hence it is a subsegment of $ [x,z] $. 
\end{enumerate}
\end{lemma}
\begin{proof}
(\ref{no tripod lemma 1}) 
Assume that $ x \not\in [y,z] $ and $ y \not\in [x,z] $. 
We will prove that $ X $ has a tripod. 
Let $ E \coloneqq [x,z] \cap [y,z] $ and $ \gamma $ the geodesic corresponding to $ [x,z] $. 
Since $ E $ is compact, there exists a point $ q \in E $ such that $ d(x,q) = \min_{p \in E}d(x,p) $.  
Note that $ q \neq x,y,z $ since $ x,y \not\in E $ and $ E \supsetneq \{z\} $. 
Take a geodesic segments $ [x,q] $ from $ [x,z] $ and two segments $ [y,q], [z,q] $ from $ [y,z] $ (Note that a geodesic segment joining two points is not necessarily unique). 
Then we may conclude that $ [x,q], [y,q] $ and $ [z,q] $ form a tripod by the choice of $ q $. 

(\ref{no tripod lemma 2})
Let $ [x,y], [y,z] $ denote the geodesic segments in $ [x,z] $ and assume that there exists another geodesic segment $ [x,y]^{\prime} $ between $ x $ and $ y $. 
If $ [x,y] \subseteq [x,y]^{\prime} $, then $ [x,y] = [x,y]^{\prime} $, which is a contradiction. 
Hence there exists $ p \in [x,y] \setminus [x,y]^{\prime} $. 
Let $ [p,z] $ be the geodesic segment in $ [x,z] $ and $ [x,z]^{\prime} $ the geodesic segment obtained by connecting $ [x,y]^{\prime} $ and $ [y,z] $. 
Then $ [p,z] \cap [x,z]^{\prime} \supseteq [y,z] \supsetneq \{z\} $. 
Hence, by (\ref{no tripod lemma 1}), we have $ p \in [x,z]^{\prime} $ or $ x \in [p,z] $. 
Each case yields a contradiction. 
\end{proof}

\begin{proof}[Proof of Theorem \ref{main 1-manifold} $ (\ref{main 1-manifold 2}) \Rightarrow (\ref{main 1-manifold 3}) $]
First we assume that $ X $ is an $ \mathbb{R} $-tree and construct a distance-preserving map from $ X $ to $ \mathbb{R} $. 
We may assume that $ X $ is not the single-point space. 
Take two distinct points $ q_{+}, q_{-} \in X $ and let $ q_{0} $ be the midpoint between $ q_{+} $ and $ q_{-} $. 

Next we show that $ [x, q_{0}] \supseteq [q_{-}, q_{0}], x \in [q_{-}, q_{+}]$,  or $[q_{0},q_{+}] \subseteq [q_{0}, x] $ for each $ x \in X $. 
In order to prove this statement,  consider geodesic segments $ [q_{-},q_{+}] $ and $ [q_{-}, x] $. 
If $ [q_{-},q_{+}] \cap [q_{-}, x] = \{q_{-}\} $, then $ [x, q_{+}] = [x, q_{-}] \cup [q_{-}, q_{+}] \supseteq [q_{-}, q_{+}] $ by Proposition \ref{R-tree equivalent conditions} and hence $ [x, q_{0}] \supseteq [q_{-}, q_{0}] $. 
When $ [q_{-},q_{+}] \cap [q_{-}, x] \supsetneq \{q_{-}\} $, we have $ x \in [q_{-}, q_{+}] $ or $ [q_{-},q_{+}] \subseteq [q_{-}, x] $ by Lemma \ref{no tripod lemma}. 
The latter condition implies $ [q_{0},q_{+}] \subseteq [q_{0}, x] $. 

Define subspaces $ X_{+} $ and $ X_{-} $ by 
\begin{align*}
X_{+} &\coloneqq \Set{x \in X | d(x,q_{-}) > d(x,q_{+})}, \\
X_{-} &\coloneqq \Set{x \in X | d(x,q_{-}) < d(x,q_{+})}. 
\end{align*}
Note that for every $ x \in X \setminus \{q_{0}\} $ we have that $ x \in X_{+} $ if and only if $ x \in [q_{0}, q_{+}] $ or $ [q_{0},q_{+}] \subseteq [q_{0}, x] $, and also we have that $ x \in X_{-} $ if and only if $ x \in [q_{-}, q_{0}] $ or $ [x,q_{0}] \supseteq [q_{-}, q_{0}] $. 

Suppose that $ x \in X \setminus (X_{+} \cup X_{-}) $, that is, $ d(x,q_{-}) = d(x,q_{+}) $. 
Then $ x \in [q_{-},q_{+}] $ and hence $ x = q_{0} $. 
Thus we obtain the decomposition $ X = X_{-} \cup \{q_{0}\} \cup X_{+} $. 
Define a map $ \phi \colon X \to \mathbb{R} $ by 
\begin{align*}
\phi(x) \coloneqq 
\begin{cases}
0 & \text{ if } x = q_{0}, \\
d(q_{0},x) & \text{ if } x \in X_{+}, \\
-d(q_{0},x) & \text{ if } x \in X_{-}. 
\end{cases}
\end{align*}

Now we will show that $ \phi $ preserves the distance. 
We will treat only the case where $ x, y \in X_{+} $ since the other cases is clear or similar and show that one of $ [q_{0},x] $ and $ [q_{0}, y] $ contains the other. 
If $ x \in [q_{0}, q_{+}] $ or $ y \in [q_{0}, q_{+}] $, then it is clear. 
Thus we may assume that $ [q_{0},q_{+}] \subseteq [q_{0},x] $ and $ [q_{0},q_{+}] \subseteq [q_{0},y] $. 
By Lemma \ref{no tripod lemma} (\ref{no tripod lemma 1}) we have $ [q_{0},y] \subseteq [q_{0},x] $ or $ [q_{0},y] \supseteq [q_{0},x] $. 
Therefore we may conclude that $ d(x,y) = |d(q_{0},x) - d(q_{0},y)| = |\phi(x) - \phi(y)| $. 
Hence $ X $ is isometric to an interval. 

Second we assume that $ X $ is not an $ \mathbb{R} $-tree and show that $ X $ is similar to $ S^{1} $. 
By Proposition \ref{R-tree equivalent conditions}, there exist three distinct points $ x,y,z \in X $ and segments $ [x,y],[x,z],[y,z] $ such that $ [x,y] \cap [y,z] = \{y\} $ and $ [x,z] \neq [x,y] \cup [y,z] $. 
We show that $ [x,z] \cap [y,z] = \{z\} $ and $ [x,y] \cap [x,z] = \{x\} $. 
To prove the former, assume $ [x,z] \cap [y,z] \supsetneq \{z\} $. 
Lemma \ref{no tripod lemma} (\ref{no tripod lemma 1}) asserts that $ x \in [y,z] $ or $ y \in [x,z] $. 
If $ x \in [y,z] $, then $ x \in [x,y] \cap [y,z] = \{y\} $. 
Hence $ x = y $, which is a contradiction. 
When $ y \in [x,z] $, apply Lemma \ref{no tripod lemma} (\ref{no tripod lemma 2}), we obtain $ [x,y] $ and $ [y,z] $ are subsegments of $ [x,z] $ and hence $ [x,z] = [x,y] \cup [y,z] $, which is again a contradiction. 
The latter case can be proved by symmetry. 

If $ X = [x,y] \cup [x,z] \cup [y,z] $, then it is easy to prove that $ X $ is similar to $ S^{1} $. 
Assume that there exists a point $ p \in X \setminus [x,y] \cup [x,z] \cup [y,z] $. 
Take a geodesic segment $ [p,z] $ and consider it together with the geodesic segments $ [x,z], [y,z] $. 
Since $ X $ has no tripod, we have $ [x,z] \cap [p,z] \supsetneq \{z\} $ or $ [y,z] \cap [p,z] \supsetneq \{z\} $. 
Without loss of generality we may assume that the latter condition holds. 
Since $ p \not\in [y,z] $, we have $ y \in [p,z] $ by Lemma \ref{no tripod lemma} (\ref{no tripod lemma 1}). 
The segments $ [y,z] $ is a subsegment of $ [p,z] $ by Lemma \ref{no tripod lemma} (\ref{no tripod lemma 2}). 
Take the subsegment $ [p,y] \subsetneq [p,z] $. 
Note that $ [p,y] \cap [y,z] = \{y\} $. 
We may deduce that $ x \in [p,y] $ in a similar way. 
Moreover, take the subsegment $ [p,x] \subsetneq [p,y] $ and we may show that $ z \in [p,x] $, which is a contradiction. 
Therefore we can conclude that $ X = [x,y] \cup [x,z] \cup [y,z] $ and it is similar to $ S^{1} $. 
\end{proof}

\section*{Acknowledgments}

The authors would like to thank the reviewers for valuable suggestions and comments, which were very helpful in making this paper more readable.

\bibliographystyle{amsplain}
\bibliography{bibfile}

\vspace{5mm}
\begin{table}[h]
\centering
\begin{tabular}{lll}
Masamichi Kuroda & \hspace*{10mm}& Shuhei Tsujie \\
Faculty of Engineering & & Department of Education \\
Nippon Bunri University & & Hokkaido University of Education \\
Oita 870-0316 & & Hokkaido 070-8621\\
Japan & & Japan  \\
kurodamm@nbu.ac.jp & &  tsujie.shuhei@a.hokkyodai.ac.jp
\end{tabular}
\end{table}

\end{document}